\documentclass[12pt,reqno]{amsart}

\usepackage{aliascnt,amsmath,amssymb,amsthm,amsfonts,enumerate,xcolor}
\usepackage{stmaryrd}
\usepackage[abbrev]{amsrefs}
\usepackage[T1]{fontenc}
\usepackage[marginparwidth=0pt,margin=24truemm]{geometry}

\definecolor{mylinkcolor}{RGB}{16, 156, 81}
\definecolor{mycitecolor}{RGB}{20, 80, 140}
\usepackage{hyperref}
\usepackage[nameinlink]{cleveref}
\hypersetup{
 setpagesize=false,
 bookmarksnumbered=true,
 bookmarksopen=true,
 hypertexnames=false,
 colorlinks=true,
 linkcolor=mylinkcolor,
 citecolor=mycitecolor,
}
\usepackage{autonum}

\theoremstyle{plain}

\crefname{theorem}{Theorem}{Theorems}
\newaliascnt{lemma}{theorem}
\newtheorem{lemma}[lemma]{Lemma}
\aliascntresetthe{lemma}
\crefname{lemma}{Lemma}{Lemmas}
\newaliascnt{proposition}{theorem}
\newtheorem{proposition}[proposition]{Proposition}
\aliascntresetthe{proposition}
\crefname{proposition}{Proposition}{Propositions}
\newaliascnt{corollary}{theorem}
\newtheorem{corollary}[corollary]{Corollary}
\aliascntresetthe{corollary}
\crefname{corollary}{Corollary}{Corollaries}
\theoremstyle{definition}
\newaliascnt{definition}{theorem}
\newtheorem{definition}[definition]{Definition}
\aliascntresetthe{definition}
\crefname{definition}{Definition}{Definitions}
\theoremstyle{remark}
\newaliascnt{remark}{theorem}
\newtheorem{remark}[remark]{Remark}
\aliascntresetthe{remark}
\crefname{remark}{Remark}{Remarks}
\theoremstyle{plain}
\newtheoremstyle{main}{\topsep}{\topsep}{\itshape}{}{\bfseries}{.}{.5em}{\thmname{#1}}
\theoremstyle{main}
\newtheorem{mainthm}{Main Theorem}
\crefformat{mainthm}{#2Main Theorem#3}
\Crefformat{mainthm}{#2Main Theorem#3}
\theoremstyle{plain}

\allowdisplaybreaks[2]
\newcommand{\Z}{\mathbb Z}
\newcommand{\Q}{\mathbb Q}
\newcommand{\cA}{\mathcal A}
\newcommand{\bq}[1]{\langle #1\rangle_q}
\newcommand{\fp}{\Z / p \Z}
\newcommand\quotient[2]{
	\mathchoice
{\text{\raise1ex\hbox{$#1$}\Big/\lower1ex\hbox{$#2$}}}
	{#1\,/\,#2}
	{#1\,/\,#2}
	{#1\,/\,#2}
}

\title{A $q$-recurrence for a finite Ap\'ery limit}
\author{Henrik Bachmann}
\address{Graduate School of Mathematics, Nagoya University, Nagoya, Japan.}
\email{henrik.bachmann@math.nagoya-u.ac.jp}
\date{\today}

\begin{document}
\begin{abstract}
The Kaneko--Zagier conjecture predicts a correspondence between finite and symmetric multiple zeta values. Under this correspondence, $\zeta(3)$ corresponds to an element $Z(3)$ defined by Bernoulli numbers. We prove a conjecture of Tasaka relating $Z(3)$ to the quotient of two solutions of a recurrence. A two-index $q$-recurrence connects this quotient to a finite harmonic $q$-series. Using a method of the author, Takeyama, and Tasaka, we obtain the algebraic and analytic limits $3Z(3)/4$ and $3\zeta(3)/4$ at roots of unity.
\end{abstract}
\maketitle

\section{Introduction}

A limit of quotients of two solutions of a second-order recurrence is called an \emph{Ap\'ery limit}, as in Ap\'ery's proof of the irrationality of $\zeta(3)$ \cite{A}. We consider the sequences $A=(A_n)_{n\geq0}$ and $C=(C_n)_{n\geq0}$ satisfying, for $n\geq1$, the recurrence\footnote{This is entry no.~28 in \cite{AESZ}.}
\begin{equation}\label{eq:panzer}
 4n^2(16n^2-1)u_{n-1}
 -(65n^4+130n^3+105n^2+40n+6)u_n
 +(n+1)^4u_{n+1}=0\,.
\end{equation}
With initial values $(A_0,A_1)=(1,6)$ and $(C_0,C_1)=(0,1)$, the sequences begin
\begin{equation}\label{eq:initials}
\begin{aligned}
 A&=(1,6,126,3948,149310,6300756,\ldots),\\
 C&=\left(0,1,\frac{173}{8},\frac{73219}{108},\frac{4922855}{192},\frac{38951317163}{36000},\ldots\right).
\end{aligned}
\end{equation}
During his work with Yeats on Martin sequences \cite{PY}, Panzer numerically observed and conjectured the formula
\begin{equation}\label{eq:real-limit}
 \lim_{n\longrightarrow\infty}\frac{C_n}{A_n}
 =\frac{\zeta(3)}7\,.
\end{equation}
This limit is proved in \cite{ST}, and we give an independent proof here as a consequence of our construction (see \cref{cor:classical-limit}).

In \cite{T}, Tasaka conjectured a finite analogue of \eqref{eq:real-limit} in the ring of ``poor man's ad\`eles''
\[
 \cA=\quotient{\prod_{p\text{ prime}}\fp\,}{\bigoplus_{p\text{ prime}}\fp}\,.
\]
Two families are identified if they agree at all but finitely many primes. For odd $k\geq3$, put
\[
 Z(k)=\left(\frac{B_{p-k}}k\bmod p\right)_p\in\cA\,,
\]
where $B_n$ is the $n$-th Bernoulli number, and we write $Z(k)_p$ for the component of $Z(k)$ at $p$. The components with $p\leq k$ may be chosen arbitrarily. We expect $Z(k)\neq0$ for every odd $k\geq3$, but this remains open for every such $k$ (see \cite{S}*{Conjecture~3}). In \cite{KZ}, Kaneko and Zagier regard $Z(3)$ as the $\cA$-analogue of $\zeta(3)$, and under their conjectural correspondence with symmetric multiple zeta values it plays the role of $\zeta(3)$ modulo $\zeta(2)$. Based on numerical experiments, Tasaka conjectured that
\[
 \left(\frac{C_{p-1}}{A_{p-1}}\right)_p=\frac{2}{21}Z(3)\,.
\]
Sun has independently communicated a proof of this identity \cite{Su}.

Our proof uses the method of the author, Takeyama, and Tasaka in \cite{BTT}, based on algebraic and analytic limits of finite harmonic $q$-series at roots of unity. In this work, we relate the algebraic limit to the recurrence sequences $A$ and $C$. For this, we consider the sums
\[
 E_m(q)=\sum_{n=1}^m\frac{(-1)^{n-1}q^{2n}}{[n]_q^3},
 \qquad [n]_q=\frac{1-q^n}{1-q}.
\]
At $q=1$, we have $E_m(1)=-\zeta(\overline{3})_m$, where $\zeta(\overline{3})_m=\sum_{n=1}^m(-1)^n/n^3$ is the truncated alternating zeta value. For odd $N\geq5$, put $q_N=e^{2\pi i/N}$ and define
\[
 \mathcal E_N=E_{(N-1)/2}(q_N)\,.
\]
As in \cite{BTT}, for a prime $p$ we identify the residue field modulo $1-q_p$ with $\fp$. The algebraic limit is obtained by reducing $\mathcal E_p$ modulo $1-q_p$ for sufficiently large primes $p$ and collecting these residues in $\cA$. For the analytic limit, we regard $\mathcal E_N$ as a complex number and let $N$ tend to infinity through odd integers. Our main result is the following.

\medskip
\begin{mainthm}\label{thm:main}
\begin{enumerate}[(i)]
\item For every prime $p\geq11$,
\begin{equation}\label{eq:strong-main}
 A_{p-1}\equiv-3p\pmod {p^2},\qquad
 C_{p-1}\equiv-\frac{2p}{21}B_{p-3}\pmod {p^2}.
\end{equation}
We also have
\[
 \mathcal E_p\equiv\frac14B_{p-3}\pmod{1-q_p}.
\]
In particular, the algebraic limit is
\[
 \left(\mathcal E_p\bmod(1-q_p)\right)_p
 =\frac{63}{8}\left(\frac{C_{p-1}}{A_{p-1}}\right)_p
 =\frac34Z(3)\,.
\]
\item The analytic limit is
\[
 \lim_{\substack{N\to\infty\\N\ \mathrm{odd}}}\mathcal E_N
 =\frac34\zeta(3)\,.
\]
\end{enumerate}
\end{mainthm}
\noindent In particular, we get $\left(C_{p-1}/A_{p-1}\right)_p=\frac{2}{21}Z(3)$, which is Tasaka's conjecture. Starting from the alternating sums $E_m(1)$, we construct the family $U_m^{(k)}\in\Q^2$ in \cref{def:classical-family}. Its diagonal, i.e. the pairs $U_n^{(n)}$, recovers $A_n$ and $C_n$ up to explicit factors. For the $q$-version in \cref{def:q-family}, we replace the integer factors in the recursion by $q$-integers. At $q=q_p$, the factors $[p]_{q_p}$ vanish, and this reduces a quotient of coordinates to $\mathcal E_p$.

The root-of-unity limits extend to every odd weight $w\geq3$ in \Cref{sec:odd-weights}, although the connection with \eqref{eq:panzer} remains special to weight three.

\section*{Acknowledgments}
The author thanks Koji Tasaka and Erik Panzer for helpful comments. This project was partially supported by JSPS KAKENHI Grant 26K22254.

\section{From the alternating sum to the Ap\'ery recurrence}\label{sec:construction}

In this section, we construct the two-index family whose diagonal gives $A_n$ and $C_n$. We begin with the alternating sums and put
\begin{equation}\label{eq:E-h}
 E_m=\sum_{n=1}^m\frac{(-1)^{n-1}}{n^3},\qquad
 h_m=\frac{(m!)^2(2m)!}{8^m},
\end{equation}
where $E_0=0$. The factor $h_m$ is chosen so that $h_mE_m$ and $h_m$ satisfy the same second-order recurrence. For $k,m\geq0$, define the coefficients
\begin{align}
 \alpha_{m+1}^{(k)}
 &=\frac{(2k+1)(2m+1)(3m^2+3m+1)}4,\label{eq:alpha}\\
 \beta_{m+1}^{(k)}
 &=\frac{m^6\bigl(4m^2-(2k+1)^2\bigr)}{16}\,.\label{eq:beta}
\end{align}
For the recursive construction, we also put
\begin{align}
 \rho_m^{(k)}
 &=\frac{(m-2k-2)(2m+2k+1)
 (m^2-2(k+1)m+4(k+1)^2)}4,\label{eq:rho}\\
 T_m^{(k)}&=\frac2{2m-2k-1}\,.
 \label{eq:T}
\end{align}
\begin{definition}\label{def:classical-family}
For $m\geq0$, set $U_m^{(0)}=h_m(E_m,1)\in\Q^2$. Define the remaining vectors recursively for $k\geq0$ by
\begin{equation}\label{eq:BM-transform}
 U_m^{(k+1)}=T_m^{(k)}
 \left(U_{m+1}^{(k)}+\rho_{m+1}^{(k)}U_m^{(k)}\right)\,.
\end{equation}
\end{definition}
These choices give the following recurrence for each fixed $k$.

\begin{lemma}\label{lem:BM}
For every fixed $k\geq0$, both coordinates of $U_m^{(k)}$ satisfy
\begin{equation}\label{eq:m-recurrence}
 U_{m+1}^{(k)}=\alpha_{m+1}^{(k)}U_m^{(k)}
                 +\beta_{m+1}^{(k)}U_{m-1}^{(k)}
 \qquad(m\geq1).
\end{equation}
\end{lemma}

\begin{proof}
For $k=0$, the second coordinate of $U_m^{(0)}$ is $h_m$, so the claimed recurrence is $h_{m+1}=\alpha_{m+1}^{(0)}h_m+\beta_{m+1}^{(0)}h_{m-1}$. Dividing by $h_m$ and using
\[
 \frac{h_{m+1}}{h_m}=\frac{(m+1)^3(2m+1)}4,
 \qquad
 \frac{h_m}{h_{m-1}}=\frac{m^3(2m-1)}4,
\]
this becomes
\[
 \frac{(m+1)^3(2m+1)}4
 =\frac{(2m+1)(3m^2+3m+1)}4+\frac{m^3(2m+1)}4\,.
\]
After cancelling $(2m+1)/4$, this is the binomial identity $(m+1)^3=m^3+3m^2+3m+1$.

For the first coordinate, subtract $E_m$ times the recurrence just proved for the second coordinate from the desired recurrence. It remains to verify
\[
 h_{m+1}(E_{m+1}-E_m)=\beta_{m+1}^{(0)}h_{m-1}(E_{m-1}-E_m)\,.
\]
Using $E_m-E_{m-1}=(-1)^{m-1}/m^3$, both sides have the common factor $(-1)^m$. Cancelling it reduces the claim to
\[
 \frac{h_{m+1}}{(m+1)^3}=\frac{\beta_{m+1}^{(0)}h_{m-1}}{m^3}\,,
\]
which holds because both sides equal $m^3(4m^2-1)h_{m-1}/16$.

We now check directly that the recurrence is preserved under the step $k\longmapsto k+1$. To keep the elimination formulas short, write
\begin{align}
 R_m^{(k)}&=\alpha_{m+1}^{(k)}+\rho_{m+1}^{(k)},\\
 \Delta_m^{(k)}&=\rho_m^{(k)}R_m^{(k)}-\beta_{m+1}^{(k)}\,.
\end{align}
Direct factorization gives
\begin{align}
 R_m^{(k)}
 &=\frac{(m+2k+2)(2m-2k-1)
 (m^2+2(k+1)m+4(k+1)^2)}4,\label{eq:R-factor}\\
 \Delta_m^{(k)}
 &=-4(k+1)^6\bigl(4m^2-(2k+1)^2\bigr)\,.
 \label{eq:Delta-factor}
\end{align}
Abbreviate $X_m=U_m^{(k+1)}/T_m^{(k)}$. Inserting \eqref{eq:m-recurrence} into \eqref{eq:BM-transform} and writing \eqref{eq:BM-transform} once more with $m$ replaced by $m-1$, we obtain the pair
\begin{equation}\label{eq:elim-system}
 X_m=R_m^{(k)}U_m^{(k)}+\beta_{m+1}^{(k)}U_{m-1}^{(k)},
 \qquad
 X_{m-1}=U_m^{(k)}+\rho_m^{(k)}U_{m-1}^{(k)}.
\end{equation}
Viewed as a linear system in the unknowns $U_m^{(k)}$ and $U_{m-1}^{(k)}$, its determinant is precisely $\Delta_m^{(k)}$. By \eqref{eq:Delta-factor} this determinant never vanishes: $2m$ is even and $2k+1$ is odd, so $4m^2\neq(2k+1)^2$, and \eqref{eq:elim-system} determines $U_m^{(k)}$ and $U_{m-1}^{(k)}$. In particular,
\[
 U_m^{(k)}=\frac{\rho_m^{(k)}X_m-\beta_{m+1}^{(k)}X_{m-1}}{\Delta_m^{(k)}}\,.
\]
On the other hand, eliminating $U_{m+1}^{(k)}$ between the two equations of \eqref{eq:elim-system} taken at $m+1$ gives $X_{m+1}=R_{m+1}^{(k)}X_m-\Delta_{m+1}^{(k)}U_m^{(k)}$. Combining the last two displays and multiplying by $T_{m+1}^{(k)}$ yields
\[
 U_{m+1}^{(k+1)}=\frac{T_{m+1}^{(k)}}{T_m^{(k)}}\left(R_{m+1}^{(k)}-\frac{\Delta_{m+1}^{(k)}}{\Delta_m^{(k)}}\rho_m^{(k)}\right)U_m^{(k+1)}+\frac{T_{m+1}^{(k)}}{T_{m-1}^{(k)}}\frac{\Delta_{m+1}^{(k)}}{\Delta_m^{(k)}}\beta_{m+1}^{(k)}U_{m-1}^{(k+1)}\,.
\]
Substituting \eqref{eq:alpha}--\eqref{eq:Delta-factor}, the two coefficients on the right become
\[
 \alpha_{m+1}^{(k+1)},\qquad \beta_{m+1}^{(k+1)}.
\]
This completes the induction.
\end{proof}

\subsection{The diagonal}
We now relate $U_n^{(n)}$ to the solutions $A_n$ and $C_n$ of \eqref{eq:panzer}. Write $D_n=U_n^{(n)}$.

\begin{proposition}\label{prop:diagonal}
For every $n\geq0$,
\begin{equation}\label{eq:diagonal-identity}
 D_n=\frac{(n!)^7}{2^n}\left(\frac{21}{4}C_n,A_n\right)\,.
\end{equation}
\end{proposition}

\begin{proof}
We first record two relations. The first equation of \eqref{eq:elim-system} at $(m,k)=(n,n-1)$ reads
\[
 \frac{D_n}{T_n^{(n-1)}}=R_n^{(n-1)}U_n^{(n-1)}+\beta_{n+1}^{(n-1)}D_{n-1}\,,
\]
and $T_n^{(n-1)}=2$, while \eqref{eq:R-factor} and \eqref{eq:beta} give $R_n^{(n-1)}=21n^3/4$ and $\beta_{n+1}^{(n-1)}=n^6(4n-1)/16$. This gives
\begin{equation}\label{eq:diag-A}
 D_n
 =\frac{21n^3}{2}U_n^{(n-1)}
 +\frac{n^6(4n-1)}8D_{n-1}
 \qquad(n\geq1).
\end{equation}
The same elimination in \eqref{eq:elim-system}, now at $(m,k)=(n+1,n-1)$, gives
\[
 \frac{U_{n+1}^{(n)}}{T_{n+1}^{(n-1)}}
 =R_{n+1}^{(n-1)}\frac{D_n}{T_n^{(n-1)}}-\Delta_{n+1}^{(n-1)}U_n^{(n-1)}\,.
\]
Here $T_{n+1}^{(n-1)}=2/3$, and \eqref{eq:R-factor} and \eqref{eq:Delta-factor} give $R_{n+1}^{(n-1)}=\frac34(3n+1)(7n^2+4n+1)$ and $\Delta_{n+1}^{(n-1)}=-12n^6(4n+1)$, so that
\begin{equation}\label{eq:diag-B}
 U_{n+1}^{(n)}
 =\frac{(3n+1)(7n^2+4n+1)}4D_n
 +8n^6(4n+1)U_n^{(n-1)}
 \qquad(n\geq1).
\end{equation}
Now take \eqref{eq:diag-A} with $n$ replaced by $n+1$, that is
\[
 D_{n+1}
 =\frac{21(n+1)^3}{2}U_{n+1}^{(n)}
 +\frac{(n+1)^6(4n+3)}8D_n\,,
\]
and use it together with \eqref{eq:diag-A} and \eqref{eq:diag-B} to eliminate $U_{n+1}^{(n)}$ and $U_n^{(n-1)}$. This gives the coordinatewise recurrence
\begin{equation}\label{eq:D-recurrence}
 \frac2{(n+1)^3}D_{n+1}
 -(65n^4+130n^3+105n^2+40n+6)D_n
 +2n^9(16n^2-1)D_{n-1}=0\,.
\end{equation}
The initial values are
\[
 D_0=(0,1),\qquad D_1=(21/8,3).
\]
Since the successive scale factors $(n!)^7/2^n$ have ratio $(n+1)^7/2$, the normalized vector sequence
\[
 \frac{2^n}{(n!)^7}D_n
\]
satisfies \eqref{eq:panzer} coordinatewise. Its first two values are $(0,1)$ and $(21/4,6)$, so \eqref{eq:initials} and uniqueness give \eqref{eq:diagonal-identity}.
\end{proof}

\subsection{The classical limit}
It is a nice feature of the construction that the diagonal identity also gives a proof of Panzer's conjectured limit.

\begin{corollary}\label{cor:classical-limit}
We have
\[
 \lim_{n\to\infty}\frac{C_n}{A_n}=\frac{\zeta(3)}7\,.
\]
\end{corollary}

\begin{proof}
By \eqref{eq:elim-system}, the transformation can also be written as
\[
 U_m^{(k+1)}=T_m^{(k)}
 \left(R_m^{(k)}U_m^{(k)}+\beta_{m+1}^{(k)}U_{m-1}^{(k)}\right)\,.
\]
For $m\geq k+1$, all three coefficients $T_m^{(k)}$, $R_m^{(k)}$, and $\beta_{m+1}^{(k)}$ are positive. Since the initial second coordinate is $h_m>0$, induction shows that the second coordinate of $U_m^{(k)}$ is positive for $m\geq k$. Therefore, we may put
\[
 x_m^{(k)}=\frac{(U_m^{(k)})_1}{(U_m^{(k)})_2}
 \qquad(m\geq k).
\]
The displayed transformation makes $x_m^{(k+1)}$ a weighted average of $x_m^{(k)}$ and $x_{m-1}^{(k)}$ with positive weights. As $x_m^{(0)}=E_m\in[0,1]$, it also proves $x_m^{(k)}\in[0,1]$ for $m\geq k$.

For each fixed $k$, these ratios tend to $3\zeta(3)/4$ as $m\to\infty$. This follows by induction on $k$ from the original transformation \eqref{eq:BM-transform}: for $m\geq2k+2$, both $T_m^{(k)}$ and $\rho_{m+1}^{(k)}$ are positive, so $x_m^{(k+1)}$ lies between $x_m^{(k)}$ and $x_{m+1}^{(k)}$. On the other hand, for $m\geq k+1$, both coefficients in \eqref{eq:m-recurrence} are positive. The closed intervals with endpoints $x_m^{(k)}$ and $x_{m+1}^{(k)}$ are therefore nested for $m\geq k$. Their common limiting value $3\zeta(3)/4$ lies in each of them. In particular, it lies between $x_n^{(n)}$ and $x_{n+1}^{(n)}$.

We next bound the distance between these endpoints. Write $a_n=4n^2(16n^2-1)$ and $P(n)=65n^4+130n^3+105n^2+40n+6$. The recurrence and the initial values give $A_n\geq2A_{n-1}>0$ for $n\geq1$. To check the induction step, use
\[
 P(n)-\frac{a_n}{2}-2(n+1)^4
 =31n^4+122n^3+95n^2+32n+4>0\,.
\]
The same recurrence gives
\[
 A_nC_{n+1}-A_{n+1}C_n
 =\frac{a_n}{(n+1)^4}
 \left(A_{n-1}C_n-A_nC_{n-1}\right)>0\,,
\]
starting from $A_0C_1-A_1C_0=1$, so $C_n/A_n$ and therefore $x_n^{(n)}=21C_n/(4A_n)$ is increasing and bounded by $1$, which gives $x_n^{(n)}-x_{n-1}^{(n-1)}\to0$.

Equation \eqref{eq:diag-A}, together with \cref{prop:diagonal}, gives
\[
 x_n^{(n-1)}-x_n^{(n)}
 =\frac{(4n-1)A_{n-1}}{4nA_n-(4n-1)A_{n-1}}
 \left(x_n^{(n)}-x_{n-1}^{(n-1)}\right)\,.
\]
The fraction is positive and at most $1$, since $A_n\geq2A_{n-1}$. Equation \eqref{eq:diag-B} puts $x_{n+1}^{(n)}$ between $x_n^{(n)}$ and $x_n^{(n-1)}$. Combining these facts with the preceding interval argument yields
\[
 0\leq\frac34\zeta(3)-x_n^{(n)}
 \leq x_n^{(n-1)}-x_n^{(n)}
 \leq x_n^{(n)}-x_{n-1}^{(n-1)}\longrightarrow0\,.
\]
Finally, $x_n^{(n)}=21C_n/(4A_n)$ gives the claimed limit.
\end{proof}

\section{The algebraic limit}\label{sec:algebraic}

We now use the same construction to prove Tasaka's conjecture. The first step is to calculate the quotient at an index where the recurrence simplifies modulo $p$. The second step is to pass to the endpoint $p-1$, where both $A_{p-1}$ and $C_{p-1}$ are divisible by $p$. This is why the final calculation must be made modulo $p^2$.

\subsection{The quotient at a quarter of the range}

Write $a_n=4n^2(16n^2-1)$ for the coefficient of $u_{n-1}$ in \eqref{eq:panzer}. Let $p\geq11$ be prime and put
\begin{equation}\label{eq:rM}
 r=\begin{cases}
 (p-1)/4,&p\equiv1\pmod4,\\
 (p+1)/4,&p\equiv3\pmod4,
 \end{cases}
 \qquad M=\frac{p-1}{2}.
\end{equation}
The choice of $r$ gives $a_r\equiv0\pmod p$. Throughout this section, rational congruences are taken in $\mathbb Z_{(p)}$, i.e. their denominators are prime to $p$.

Write
\begin{equation}\label{eq:P}
 P(n)=65n^4+130n^3+105n^2+40n+6\,.
\end{equation}
The middle term in \eqref{eq:panzer} is then $-P(n)u_n$. We first check the denominators of the quantities that will be reduced modulo $p$. Notice that the terms $C_n$ need not be integers, for example $C_2=173/8$.

\begin{lemma}\label{lem:integrality}
Let $u=A$ or $u=C$. Then $u_n$ is $p$-integral for $0\leq n\leq p-1$. Moreover, both coordinates of $U_m^{(k)}$ are $p$-integral for all $(m,k)$ in the set
\[
 \mathcal T=\{(m,k):0\leq k\leq r,\ 0\leq m\leq 2r-k\}\,.
\]
\end{lemma}

\begin{proof}
Solving \eqref{eq:panzer} for $u_{n+1}$ gives
\[
 u_{n+1}=\frac{P(n)u_n-a_nu_{n-1}}{(n+1)^4}\,.
\]
For $n\leq p-2$ the denominator $(n+1)^4$ is prime to $p$, so the first assertion follows by induction from $u_0,u_1\in\Z$.

For the second assertion, let $k=0$. Since $8,1,\ldots,2r$ are prime to $p$, both $h_m$ and $E_m$ are $p$-integral for $0\leq m\leq2r$, and so is $U_m^{(0)}$.

Now let $0\leq k\leq r-1$ and $0\leq m\leq 2r-k-1$, so that $(m,k+1)\in\mathcal T$. Then $(m+1,k),(m,k)\in\mathcal T$, so the induction hypothesis applies to $U_{m+1}^{(k)}$ and $U_m^{(k)}$. Moreover, $\rho_{m+1}^{(k)}\in\frac14\Z\subset\Z_{(p)}$. It remains only to check the denominator of $T_m^{(k)}=2/(2(m-k)-1)$. Since $-k\leq m-k\leq 2r-2k-1$, we get
\[
 1\leq\bigl|2(m-k)-1\bigr|\leq\max(2k+1,4r-4k-3)\leq 4r-3<p\,,
\]
the last inequality because $4r-3$ equals $p-4$ or $p-2$ according as $p=4r+1$ or $p=4r-1$. The claim follows by induction on $k$.
\end{proof}

\begin{proposition}\label{prop:endpoint}
With $r,M$ as in \eqref{eq:rM}, one has $A_r\not\equiv0\pmod p$ and
\begin{equation}\label{eq:quarter-E}
 \frac{21}{4}\frac{C_r}{A_r}\equiv E_M\pmod p.
\end{equation}
\end{proposition}

\begin{proof}
Every vector $U_m^{(k)}$ occurring below is indexed by a pair in the triangle $\mathcal T$ of \cref{lem:integrality} and is therefore $p$-integral. This allows us to discard the terms divisible by $p$.

Suppose first that $p=4r+1$. For $0\leq j\leq r-1$, apply \eqref{eq:BM-transform} with
\[
 m=r+j,\qquad k=r-j-1.
\]
The second factor of $\rho_{m+1}^{(k)}$ is
\[
 2(m+1)+2k+1=4r+1=p\,,
\]
whereas the denominator of $T_m^{(k)}$ is $4j+1$, which is nonzero modulo $p$. Therefore the term $\rho_{m+1}^{(k)}U_m^{(k)}$ in \eqref{eq:BM-transform} vanishes modulo $p$. Since $2r=M$, iteration gives
\[
 U_r^{(r)}\equiv
 \left(\prod_{j=0}^{r-1}\frac2{4j+1}\right)U_M^{(0)}\pmod p.
\]

Now suppose that $p=4r-1$. At $(m,k)=(r,r-1)$, we have
\[
 T_r^{(r-1)}R_r^{(r-1)}=\frac{21r^3}{2},
 \qquad
 T_r^{(r-1)}\beta_{r+1}^{(r-1)}
 =\frac{r^6(4r-1)}8\equiv0\pmod p.
\]
Combining \eqref{eq:m-recurrence} with \eqref{eq:BM-transform} therefore gives
\[
 U_r^{(r)}\equiv\frac{21r^3}{2}U_r^{(r-1)}\pmod p.
\]
The scalar $21r^3/2$ is nonzero modulo $p$ because $p\geq11$. This is the only place in this proof where this hypothesis is used. Notice that the exclusion of $p=7$ is necessary: in that case $r=2$ and $A_2=126\equiv0\pmod 7$, so \eqref{eq:quarter-E} fails. For $0\leq j\leq r-2$, use \eqref{eq:BM-transform} with
\[
 m=r+j,\qquad k=r-j-2.
\]
This time the second factor of $\rho_{m+1}^{(k)}$ is $4r-1=p$, and the denominator of $T_m^{(k)}$ is $4j+3$. Since $2r-1=M$, it follows that
\[
 U_r^{(r)}\equiv\frac{21r^3}{2}
 \left(\prod_{j=0}^{r-2}\frac2{4j+3}\right)U_M^{(0)}\pmod p.
\]

In both cases, $U_r^{(r)}$ is a nonzero scalar multiple of $U_M^{(0)}$ modulo $p$. By \cref{def:classical-family,prop:diagonal}, these vectors are
\[
 \frac{(r!)^7}{2^r}\left(\frac{21}{4}C_r,A_r\right)
 \quad\text{and}\quad h_M(E_M,1)\,,
\]
respectively. Both displayed scalar factors are nonzero modulo $p$. This shows that the second coordinate of $U_M^{(0)}$ is nonzero, so the proportionality above implies the same for $U_r^{(r)}$, and therefore $A_r\not\equiv0\pmod p$. Taking coordinate ratios proves \eqref{eq:quarter-E}.
\end{proof}

\begin{remark}\label{rem:endpoint-vs-q}
The proof of \cref{prop:endpoint} is the $q=1$ version of the cancellation used in the proof of \cref{prop:q-root} below. There, $U_r^{(r)}(q)$ is expanded in terms of vectors $U_j^{(0)}(q)$. We give the two arguments separately because \cref{prop:endpoint} is needed for the congruences modulo $p^2$ at the end of this section, and because the proof of \cref{prop:q-root} uses the formula above when $p=4r-1$.
\end{remark}

It remains to evaluate the half-sum. We use the Bernoulli polynomials defined by
\[
 \frac{te^{xt}}{e^t-1}=\sum_{n\geq0}B_n(x)\frac{t^n}{n!},
 \qquad B_n=B_n(0).
\]
\begin{lemma}\label{lem:half-sum}
For every prime $p\geq7$ and $M=(p-1)/2$,
\begin{equation}\label{eq:E-B}
 E_M\equiv\frac14B_{p-3}\pmod p.
\end{equation}
\end{lemma}

\begin{proof}
Write $H_m^{(3)}=\sum_{j=1}^mj^{-3}$ and $S=\sum_{j=1}^{p-1}(-1)^{j-1}j^{-3}$.

Pairing the terms indexed by $j$ and $p-j$ gives $S\equiv2E_M$, since $(-1)^{p-j-1}(p-j)^{-3}\equiv(-1)^{j-1}j^{-3}\pmod p$ shows that the two halves of $S$ agree modulo $p$.

Separating the even terms gives
\[
 S=H_{p-1}^{(3)}-2\sum_{i=1}^{M}\frac1{(2i)^3}
  =H_{p-1}^{(3)}-\frac14H_M^{(3)}\,,
\]
and $H_{p-1}^{(3)}\equiv0\pmod p$ because $j^{-3}\equiv j^{p-4}$ and $p-1\nmid p-4$ for $p\geq7$. Combining the two congruences gives $E_M\equiv-H_M^{(3)}/8$.

The Seki--Bernoulli formula $\sum_{j=0}^{n-1}j^k=\bigl(B_{k+1}(n)-B_{k+1}\bigr)/(k+1)$ with $k=p-4$ and $n=M+1$ gives
\[
 H_M^{(3)}
 \equiv\frac{B_{p-3}(M+1)-B_{p-3}}{p-3}\pmod p.
\]
For $1\leq i\leq p-3$, we have $p-1\nmid i$, while $B_0=1$. By von Staudt--Clausen, $B_i$ is therefore $p$-integral for $0\leq i\leq p-3$, and the polynomial $B_{p-3}(x)$ has $p$-integral coefficients and may be evaluated modulo $p$. Now $M+1\equiv1/2\pmod p$ and
\[
B_{p-3}(1/2)=(2^{4-p}-1)B_{p-3}\equiv7B_{p-3}\pmod p.
\]
Here the congruence follows from Fermat's theorem. So the numerator in the Seki--Bernoulli formula is $6B_{p-3}$ and its denominator is $p-3\equiv-3$, so $H_M^{(3)}\equiv-2B_{p-3}\pmod p$. This proves \eqref{eq:E-B}.
\end{proof}

By \cref{prop:endpoint,lem:half-sum},
\begin{equation}\label{eq:lambda}
 \lambda:=\frac{C_r}{A_r}\equiv\frac1{21}B_{p-3}\pmod p.
\end{equation}

\subsection{The endpoint normalization}
It remains to pass from the index $r$ to $p-1$. We first record two facts needed at the index $p-1$.

\begin{lemma}\label{lem:A-endpoint}
The solution $A$ has the integral binomial representation
\begin{equation}\label{eq:A-double-sum}
 A_n=\sum_{i,j=0}^n
 \binom ni^2\binom nj^2\binom{i+j}{n}^2\,.
\end{equation}
In particular $A_n\in\mathbb Z$ for every $n\geq0$, and, for every prime $p\geq5$,
\begin{equation}\label{eq:A-minus3}
 A_{p-1}\equiv-3p\pmod {p^2}.
\end{equation}
\end{lemma}

\begin{proof}
The binomial representation \eqref{eq:A-double-sum} is given in \cite{AESZ}*{entry no.~28}, and it shows that every $A_n$ is an integer. At $n=p-1$, the terms with $i+j<p-1$ vanish. If $i+j=p-1+t$ with $1\leq t\leq p-1$, then Lucas' theorem gives
\[
 \binom{p-1+t}{p-1}
 =\binom{p-1+t}{t}
 \equiv\binom10\binom{t-1}t=0\pmod p.
\]
So the terms with $i+j>p-1$ vanish modulo $p^2$, only $i+j=p-1$ remains, and we get
\[
 A_{p-1}\equiv\sum_{i=0}^{p-1}\binom{p-1}{i}^4\pmod {p^2}.
\]
Writing $H_i=\sum_{j=1}^i1/j$, we have
\[
 \binom{p-1}{i}\equiv(-1)^i(1-pH_i)\pmod {p^2}.
\]
This gives
\[
 A_{p-1}\equiv p-4p\sum_{i=0}^{p-1}H_i\pmod {p^2}.
\]
Finally
\[
 \sum_{i=0}^{p-1}H_i
 =\sum_{j=1}^{p-1}\frac{p-j}{j}\equiv1\pmod p,
\]
which proves \eqref{eq:A-minus3}.
\end{proof}

Panzer and Yeats also obtain \eqref{eq:A-minus3} from their factorial formula for the Martin sequence \cite{PY}*{Lemma~6.33}. The argument above instead derives it directly from \eqref{eq:A-double-sum}.

\begin{lemma}\label{lem:W}
For the two solutions $A$ and $C$, let
\[
 W_n=A_nC_{n+1}-A_{n+1}C_n\,.
\]
Then
\begin{equation}\label{eq:W-closed}
 W_n=\frac{(4n+1)\binom{4n}{2n}\binom{2n}{n}}{(n+1)^4}\,.
\end{equation}
For every prime $p\geq5$,
\begin{equation}\label{eq:W-unit}
 p^2W_{p-1}\equiv-3\pmod p.
\end{equation}
\end{lemma}

\begin{proof}
Write \eqref{eq:panzer} as $(n+1)^4u_{n+1}=P(n)u_n-a_nu_{n-1}$ and apply it to both $A$ and $C$:
\[
 \begin{aligned}
 (n+1)^4W_n
 &=A_n\bigl(P(n)C_n-a_nC_{n-1}\bigr)-\bigl(P(n)A_n-a_nA_{n-1}\bigr)C_n\\
 &=a_n\bigl(A_{n-1}C_n-A_nC_{n-1}\bigr)=a_nW_{n-1}.
 \end{aligned}
\]
So $W_n$ satisfies the first-order recurrence
\[
 W_n=\frac{a_n}{(n+1)^4}W_{n-1},\qquad
 W_0=A_0C_1-A_1C_0=1.
\]
The expression on the right of \eqref{eq:W-closed} equals $1$ at $n=0$, and the ratio of its values at $n$ and $n-1$ is
\[
 \frac{4n+1}{4n-3}\cdot\frac{4n(4n-1)(4n-2)(4n-3)}{\bigl((2n)(2n-1)\bigr)^2}
 \cdot\frac{(2n)(2n-1)}{n^2}\cdot\frac{n^4}{(n+1)^4}
 =\frac{4n^2(16n^2-1)}{(n+1)^4}
 =\frac{a_n}{(n+1)^4}\,.
\]
Therefore, it satisfies the same recurrence and initial condition as $W_n$.

We use the following two congruences, obtained after dividing out one factor of $p$:
\begin{equation}\label{eq:binomial-units}
 \frac1p\binom{2p-2}{p-1}\equiv-1\pmod p,
 \qquad
 \frac1p\binom{4p-4}{2p-2}\equiv-1\pmod p.
\end{equation}
For the first, use
\[
 \binom{2p-2}{p-1}=\frac{p}{2p-1}\binom{2p-1}{p-1},
 \qquad \binom{2p-1}{p-1}\equiv1\pmod p.
\]
For the second, use the product form of the binomial coefficient. Its numerator runs from $2p-1$ to $4p-4$, and its denominator from $1$ to $2p-2$. Accounting for the prefactor $1/p$, the multiples of $p$ contribute $(2p)(3p)/(p\cdot p)=6$. Reducing all remaining factors modulo $p$ gives
\[
 \frac1p\binom{4p-4}{2p-2}
 \equiv
 6\frac{(p-1)(p-1)!(p-4)!}{(p-1)!(p-2)!}\equiv-1\pmod p.
\]
Substituting \eqref{eq:binomial-units} into \eqref{eq:W-closed} at $n=p-1$ proves \eqref{eq:W-unit}.
\end{proof}

\subsection{Reversing the recurrence}
Let $s=p-r$. Since $a_n=4n^2(4n-1)(4n+1)$, the only indices $1\leq n\leq p-1$ for which $a_n\equiv0\pmod p$ are $r$ and $s$. Moreover, $r<s\leq p-2$ because $2r<p$ and $r\geq3$.

\begin{lemma}\label{lem:collapse}
Modulo $p$, one has
\[
 A_n\equiv C_n\equiv0\qquad(s\leq n\leq p-1),
\]
and
\begin{equation}\label{eq:freeze}
 C_n\equiv\lambda A_n\qquad(r\leq n\leq p-1),
\end{equation}
where $\lambda$ is defined in \eqref{eq:lambda}.
\end{lemma}

\begin{proof}
By \cref{prop:endpoint}, $A_r\not\equiv0\pmod p$, so $\lambda$ is defined. The sequence $C-\lambda A$ satisfies \eqref{eq:panzer} and vanishes at $r$ modulo $p$. Let $u$ be any solution with $u_r\equiv0$. At $n=r$ the coefficient $a_r$ vanishes modulo $p$, so \eqref{eq:panzer} reduces to $(r+1)^4u_{r+1}\equiv0$, and $r+1\not\equiv0$ gives $u_{r+1}\equiv0$. For $r<n\leq p-2$ the coefficient $(n+1)^4$ is again nonzero modulo $p$, so $u_{n-1}\equiv u_n\equiv0$ forces $u_{n+1}\equiv0$. Induction gives $u_n\equiv0\pmod p$ for $r\leq n\leq p-1$. Applying this to $C-\lambda A$ proves \eqref{eq:freeze}.

On the other hand, \cref{lem:A-endpoint} gives $A_{p-1}\equiv0\pmod p$. For $s<n\leq p-1$ the coefficient $a_n$ is nonzero modulo $p$, so \eqref{eq:panzer} determines $u_{n-1}$ from $u_n$ and $u_{n+1}$. At $n=p-1$ the last term of \eqref{eq:panzer} is $p^4A_p$, which vanishes modulo $p$ because $A_p\in\mathbb Z$. Since $a_{p-1}\equiv a_1=60\not\equiv0\pmod p$, the recurrence gives $A_{p-2}\equiv0$. Running the recurrence backwards from $n=p-2$ down to $n=s+1$ now gives $A_n\equiv0$ for $s\leq n\leq p-1$. As $r\leq s$, equation \eqref{eq:freeze} gives the same for $C$.
\end{proof}

For $u=A$ or $u=C$, the quantities $u_{p-1}/p$ and $p^3u_p$ are $p$-integral. To see this, notice that $u_0,\ldots,u_{p-1}$ are $p$-integral by \cref{lem:integrality}, while \cref{lem:collapse} gives $p\mid u_{p-2}$ and $p\mid u_{p-1}$. The recurrence at $n=p-1$ reads
\[
 p^4u_p=P(p-1)u_{p-1}-a_{p-1}u_{p-2}\,.
\]
Its right-hand side is divisible by $p$, so $p^3u_p$ is $p$-integral. The term $u_p$ itself need not be $p$-integral.

To write the reversed recurrence in the form \eqref{eq:panzer}, define
\begin{equation}\label{eq:G-hat}
 G_0=1,\qquad
 G_k=\prod_{j=1}^k\frac{a_{p-j}}{j^4},\qquad
 \widehat u_k=G_k\frac{u_{p-1-k}}p
 \quad(0\leq k\leq r).
\end{equation}
\begin{lemma}\label{lem:reflection}
Let $u=A$ or $u=C$. Then, for $0\leq k\leq r$,
\begin{equation}\label{eq:hat-solution}
 \widehat u_k
 \equiv\left(\frac{u_{p-1}}p\right)A_k-p^3u_pC_k\pmod p.
\end{equation}
\end{lemma}

\begin{proof}
For $k<r$, $G_k$ is nonzero modulo $p$, and $u_{p-1-k}$ is divisible by $p$ by \cref{lem:collapse}. The product $G_r$ contains the factor $a_s$ exactly once, and since one of $4s-1$ and $4s+1$ is $3p$, the number $G_r$ is divisible by $p$ but not by $p^2$. Therefore, $\widehat u_k$ is $p$-integral for $0\leq k\leq r$, and $G_r/p\not\equiv0\pmod p$.

Apply \eqref{eq:panzer} with $n=p-1-k$ and multiply by $G_k/p$. Since
\[
 \frac{G_{k+1}}{G_k}=\frac{a_{p-k-1}}{(k+1)^4},
 \qquad a_{p-k}\equiv a_k\pmod p,
 \qquad P(-1-k)=P(k),
\]
we obtain
\[
 a_k\widehat u_{k-1}-P(k)\widehat u_k
 +(k+1)^4\widehat u_{k+1}\equiv0\pmod p
\]
for $1\leq k<r$, which is \eqref{eq:panzer} at $n=k$. The original recurrence at $n=p-1$, divided by $p$, gives
\[
 \widehat u_1-6\widehat u_0+p^3u_p\equiv0\pmod p.
\]
Since $\widehat u_0=u_{p-1}/p$, the sequence
\[
 \left(\frac{u_{p-1}}p\right)A_k-p^3u_pC_k
\]
has the same first two values as $(\widehat u_k)$ and satisfies the same recurrence. This proves \eqref{eq:hat-solution}.
\end{proof}

By \cref{lem:A-endpoint},
\[
 \frac{A_{p-1}}p\equiv-3,\qquad p^3A_p\equiv0\pmod p.
\]
The second congruence uses $A_p\in\mathbb Z$. At $n=p-1$, the definition of $W_{p-1}$ gives
\[
 W_{p-1}=A_{p-1}C_p-A_pC_{p-1}\,.
\]
Multiplying by $p^2$ and reducing modulo $p$ gives
\[
 -3\equiv p^2W_{p-1}
 \equiv\frac{A_{p-1}}p\,p^3C_p
       -p^3A_p\frac{C_{p-1}}p
 \equiv-3p^3C_p\pmod p,
\]
where we used \cref{lem:W}. This gives $p^3C_p\equiv1\pmod p$, and \cref{lem:reflection} gives
\[
 \widehat A_k\equiv-3A_k,\qquad
 \widehat C_k\equiv
 \left(\frac{C_{p-1}}p\right)A_k-C_k\pmod p.
\]

Since $r\leq s-1\leq p-1$, \eqref{eq:freeze} applies at $n=s-1=p-1-r$. Together with $G_r/p\not\equiv0\pmod p$, this gives
\[
 \widehat C_r\equiv\lambda\widehat A_r\pmod p.
\]
Substituting the reflection formulas into $\widehat C_r\equiv\lambda\widehat A_r$ and using $C_r\equiv\lambda A_r\pmod p$ gives
\[
 \left(\frac{C_{p-1}}p-\lambda\right)A_r
 \equiv-3\lambda A_r\pmod p.
\]
Since $A_r\not\equiv0\pmod p$, it follows that $C_{p-1}/p\equiv-2\lambda\pmod p$. By \eqref{eq:lambda}, this is equivalent to
\[
 C_{p-1}\equiv-\frac{2p}{21}B_{p-3}\pmod {p^2}.
\]
Together with $A_{p-1}/p\equiv-3\pmod p$, this also gives
\[
 \frac{C_{p-1}/p}{A_{p-1}/p}
 \equiv\frac{-2\lambda}{-3}
 =\frac23\lambda
 \equiv\frac{2}{63}B_{p-3}\pmod p.
\]
This proves the congruences for $A_{p-1}$ and $C_{p-1}$ in part~\textup{(i)} of the \cref{thm:main}.

\section{The \texorpdfstring{$q$}{q}-recurrence and the analytic limit}
\label{sec:cyclotomic}

In the proof of \cref{prop:endpoint}, a factor equal to $p$ made most terms vanish modulo $p$. We now replace integer factors by $q$-integers, so that the same terms vanish exactly when $q$ is a primitive $p$-th root of unity. The specialization of this $q$-family at $q=1$ recovers the preceding construction.

\subsection{The construction}
Some factors in $\rho_m^{(k)}$ and $T_m^{(k)}$ are negative. We extend the positive $q$-integers by their sign: for $a\in\mathbb Z$, define
\[
 \bq a=
 \begin{cases}
 [a]_q,&a>0,\\
 0,&a=0,\\
 -[-a]_q,&a<0.
\end{cases}
\]
At $q=1$, $\bq a=a$. For $m\geq1$, set $[m]_q!=\prod_{j=1}^m[j]_q$, and set $[0]_q!=1$. For $m\geq0$, put
\begin{align}
 E_m(q)&=\sum_{n=1}^m\frac{(-1)^{n-1}q^{2n}}{[n]_q^3},\label{eq:q-E}\\
 h_m(q)&=\frac{[m]_q!^2[2m]_q!}{[2]_q^{3m}}\,.\nonumber
\end{align}
For $k,m\geq0$, define
\begin{align}
 T_m^{(k)}(q)
 &=\frac{\bq 2}{\bq{2m-2k-1}},\label{eq:q-T}\\
 \rho_m^{(k)}(q)
 &=\frac{\bq{m-2k-2}\bq{2m+2k+1}}{\bq 2^2}
 \left(\bq m\bq{m-2k-2}+\bq{2k+2}^2\right)\,.\label{eq:q-rho}
\end{align}
\begin{definition}\label{def:q-family}
For $m\geq0$, set $U_m^{(0)}(q)=h_m(q)(E_m(q),1)$. Define the remaining pairs recursively for $k\geq0$ by
\begin{equation}\label{eq:q-BM}
 U_m^{(k+1)}(q)
 =T_m^{(k)}(q)
\left(U_{m+1}^{(k)}(q)
+\rho_{m+1}^{(k)}(q)U_m^{(k)}(q)\right)\,.
\end{equation}
\end{definition}
For example,
\[
 U_0^{(0)}(q)=(0,1),\qquad
 U_1^{(0)}(q)=\frac1{(1+q)^2}(q^2,1)
 =\bigl(q^2-2q^3+3q^4-\cdots,
        1-2q+3q^2-\cdots\bigr)\,.
\]
By induction on $k$, the coordinates of $U_m^{(k)}(q)$ are rational functions of $q$. Every denominator introduced above has a nonzero constant term, so $U_m^{(k)}(q)\in\Q(q)^2\cap\Q\llbracket q\rrbracket^2$.

At $q=1$, the quantities in \eqref{eq:q-T} and \eqref{eq:q-rho} are $T_m^{(k)}$ and $\rho_m^{(k)}$. Also $U_m^{(0)}(1)=U_m^{(0)}$, so the notation is compatible with the classical family in \cref{def:classical-family}: throughout the paper, $U_m^{(k)}$ means $U_m^{(k)}(1)$. It follows from \cref{prop:diagonal} that
\[
 U_n^{(n)}(1)
 =\frac{(n!)^7}{2^n}\left(\frac{21}{4}C_n,A_n\right)\,.
\]

Notice that the construction uses only the recursion in $k$. The second-order recurrence in $m$ from \cref{lem:BM} was needed to identify the diagonal at $q=1$.

\subsection{Evaluation at a root of unity}
We expand the diagonal by repeatedly applying the $q$-recursion.

\begin{proposition}\label{prop:q-root}
Let $p\geq11$ be prime, let $q_p=e^{2\pi i/p}$, and put $r=\lfloor(p+1)/4\rfloor$. Then the second coordinate of $U_r^{(r)}(q_p)$ is nonzero and
\[
 \frac{(U_r^{(r)}(q_p))_1}
      {(U_r^{(r)}(q_p))_2}
 =E_{(p-1)/2}(q_p)\,.
\]
\end{proposition}

\begin{proof}
Repeated substitution in \eqref{eq:q-BM} expresses $U_r^{(r)}(q)$ as a sum of scalar multiples of vectors $U_j^{(0)}(q)$. At each step, choosing the first summand increases $m$ by $1$, while choosing the second leaves $m$ unchanged and contributes $\rho_{m+1}^{(k)}(q)$. After $\ell$ substitutions, a term in which the second summand has been chosen $t$ times is a scalar multiple of
\[
 U_{r+\ell-t}^{(r-\ell)}(q)\,.
\]
For the next substitution we have $m=r+\ell-t$ and $k=r-\ell-1$. The denominator of $T_m^{(k)}(q)$ is therefore
\[
 \bq{2m-2k-1}=\bq{4\ell-2t+1}\,.
\]
The integer $4\ell-2t+1$ lies between $1$ and $4r-3<p$, so none of the denominators of the $T$-factors vanishes at $q=q_p$. The remaining denominators are powers of $[2]_q$ and the factors $[j]_q$ in the terminal sums $E_m(q)$, where $m\leq2r<p$, so the whole expansion is regular at $q=q_p$ and we may specialize it term by term. If the second summand is chosen in the next substitution, the second factor in $\rho_{m+1}^{(k)}(q)$ is
\begin{equation}\label{eq:q-second-factor}
 \bq{2(m+1)+2k+1}=\bq{4r+1-2t}\,.
\end{equation}
After all $r$ substitutions, such a term is a scalar multiple of $U_{2r-t}^{(0)}(q_p)$.

Suppose first that $p=4r+1$. The first occurrence of the second summand contributes the factor $[p]_{q_p}=0$ by \eqref{eq:q-second-factor}. Therefore, only the term obtained by always choosing the first summand survives, and it is a scalar multiple of $U_{2r}^{(0)}(q_p)=U_{(p-1)/2}^{(0)}(q_p)$.

Now suppose that $p=4r-1$. The term obtained by always choosing the first summand is zero because it is a scalar multiple of $U_{2r}^{(0)}(q_p)=0$, since $h_{2r}(q_p)$ contains $[p]_{q_p}$ and $E_{2r}(q_p)$ has no pole. A term for which the second summand is chosen at least twice is also zero: at the second such choice, $t=1$ in \eqref{eq:q-second-factor}, so the coefficient contains the vanishing factor~$[p]_{q_p}$. The only possibly nonzero terms are therefore those for which the second summand is chosen exactly once, and all of them are scalar multiples of $U_{2r-1}^{(0)}(q_p)=U_{(p-1)/2}^{(0)}(q_p)$. So in both cases,
\begin{equation}\label{eq:q-proportional}
 U_r^{(r)}(q_p)=\gamma_p U_{(p-1)/2}^{(0)}(q_p)
\end{equation}
for some $\gamma_p\in\Q(q_p)$.

It remains to prove that $\gamma_p\neq0$. Put
\[
 \mathcal O_p=\mathbb Z[q_p]_{(1-q_p)}\,.
\]
Its maximal ideal is generated by $1-q_p$, and $\mathcal O_p/(1-q_p)\simeq\mathbb F_p$. For every integer $a$,
\[
 \bq a\big|_{q=q_p}\equiv a\pmod{1-q_p}.
\]
So each denominator above has nonzero residue in $\mathbb F_p$ and is a unit in $\mathcal O_p$, and therefore the scalar coefficients in the iterated expansion, and in particular $\gamma_p$, lie in $\mathcal O_p$ and reduce to their values at $q=1$. The element $h_{(p-1)/2}(q_p)$ is a unit in $\mathcal O_p$ as well, because all factors in its $q$-factorials have indices between $1$ and $p-1$. If $p=4r+1$, there is only one surviving term, and its coefficient is
\[
 \gamma_p=\prod_{j=0}^{r-1}\frac{[2]_{q_p}}{[4j+1]_{q_p}}\neq0\,.
\]

Let $p=4r-1$. Here the coefficients of the terms in which the second summand is chosen exactly once could cancel. To show that they do not, let $\overline\gamma_p\in\mathbb F_p$ be the residue of $\gamma_p$. Reducing \eqref{eq:q-proportional} modulo $1-q_p$ gives
\[
 U_r^{(r)}\equiv
 \overline\gamma_p U_{(p-1)/2}^{(0)}\pmod p.
\]
The explicit congruence obtained in the case $p=4r-1$ of the proof of \cref{prop:endpoint} is
\[
 U_r^{(r)}\equiv\frac{21r^3}{2}
 \left(\prod_{j=0}^{r-2}\frac2{4j+3}\right)U_{(p-1)/2}^{(0)}\pmod p.
\]
Comparing second coordinates and cancelling $h_{(p-1)/2}\not\equiv0\pmod p$ gives
\[
 \overline\gamma_p=
 \frac{21r^3}{2}\prod_{j=0}^{r-2}\frac2{4j+3}
 \not\equiv0\pmod p.
\]
To see this, notice that $p\nmid21r$ because $p\geq11$ and $0<r<p$, while $4j+3\leq4r-5=p-4$ throughout the product. This shows $\gamma_p\neq0$ in both cases. Finally, $U_{(p-1)/2}^{(0)}(q_p) =h_{(p-1)/2}(q_p)(E_{(p-1)/2}(q_p),1)$, and since both $\gamma_p$ and $h_{(p-1)/2}(q_p)$ are nonzero, the second coordinate in \eqref{eq:q-proportional} is nonzero. Taking the ratio of the two coordinates in \eqref{eq:q-proportional} gives the result.
\end{proof}

\subsection{The two limits}
We now prove the remaining assertions of the \cref{thm:main}. Since $q_p\equiv1$ and $[j]_{q_p}\equiv j$ modulo $1-q_p$,
\begin{equation}\label{eq:q-half-reduction}
 E_{(p-1)/2}(q_p)
 \equiv E_{(p-1)/2}
 \equiv\frac14B_{p-3}\pmod{1-q_p}
\end{equation}
by \cref{lem:half-sum}. Together with the quotient congruence in \Cref{sec:algebraic}, this proves the equality of algebraic limits and completes the proof of part~\textup{(i)} of the \cref{thm:main}.

For odd $N$ and $1\leq j\leq(N-1)/2$, the sine estimate from \cite{BTT} gives
\[
 |[j]_{q_N}|
 =\frac{\sin(\pi j/N)}{\sin(\pi/N)}
 \geq\frac{2j}{\pi}\,.
\]
Extend the summand by zero for $j>(N-1)/2$. For each fixed $j$, $q_N^{2j}/[j]_{q_N}^3\longrightarrow1/j^3$. Dominated convergence then gives
\[
 \lim_{\substack{N\to\infty\\N\ \mathrm{odd}}}
 E_{(N-1)/2}(q_N)
 =\sum_{j\geq1}\frac{(-1)^{j-1}}{j^3}
 =\frac34\zeta(3)\,.
\]
This proves part~\textup{(ii)} of the \cref{thm:main}.

\section{General odd weights}\label{sec:odd-weights}

Notice that the cancellation at roots of unity only depends on the transformation coefficients and the common second coordinate, and not on the weight of the alternating sum. We may therefore replace the alternating cubic sum by an alternating sum of any odd weight. Let $w\geq3$ be odd and put
\begin{equation}\label{eq:odd-E}
 E_m^{(w)}(q)
 =\sum_{n=1}^m\frac{(-1)^{n-1}q^{2n}}{[n]_q^w},
 \qquad E_m^{(w)}=E_m^{(w)}(1).
\end{equation}
For odd $N\geq5$, define
\begin{equation}\label{eq:odd-normalization}
 \mathcal E_N^{(w)}=E_{(N-1)/2}^{(w)}(q_N)\,.
\end{equation}
In particular, $E_m^{(3)}(q)=E_m(q)$ and $\mathcal E_N^{(3)}=\mathcal E_N$.

For each odd $w$, define the corresponding $q$-family by setting
\begin{equation}\label{eq:odd-family-initial}
 U_m^{(0,w)}(q)=h_m(q)\bigl(E_m^{(w)}(q),1\bigr)
\end{equation}
and
\begin{equation}\label{eq:odd-family-recursion}
 U_m^{(k+1,w)}(q)
 =T_m^{(k)}(q)
 \left(U_{m+1}^{(k,w)}(q)
 +\rho_{m+1}^{(k)}(q)U_m^{(k,w)}(q)\right)\,.
\end{equation}
For $w=3$, this is the family from \Cref{sec:cyclotomic}.

\pagebreak[3]
\begin{proposition}\label{prop:odd-weights}
Let $w\geq3$ be odd.
\begin{enumerate}[(i)]
\item If $p\geq11$ is prime and $r=\lfloor(p+1)/4\rfloor$, then the second coordinate of $U_r^{(r,w)}(q_p)$ is nonzero and
\begin{equation}\label{eq:odd-root}
 \frac{(U_r^{(r,w)}(q_p))_1}
      {(U_r^{(r,w)}(q_p))_2}
 =E_{(p-1)/2}^{(w)}(q_p)\,.
\end{equation}
\item If $p>w$ is prime, then
\begin{equation}\label{eq:odd-algebraic}
 \mathcal E_p^{(w)}
 \equiv\left(1-2^{1-w}\right)Z(w)_p\pmod{1-q_p}.
\end{equation}
Changing finitely many components does not change an element of $\cA$, so the algebraic limit is $(1-2^{1-w})Z(w)$.
\item The analytic limit is
\begin{equation}\label{eq:odd-analytic}
 \lim_{\substack{N\to\infty\\N\ \mathrm{odd}}}
 \mathcal E_N^{(w)}
 =\left(1-2^{1-w}\right)\zeta(w)\,.
\end{equation}
\end{enumerate}
\end{proposition}

\begin{proof}
The expansion used in the proof of \cref{prop:q-root} depends only on the factors $T_m^{(k)}(q)$ and $\rho_m^{(k)}(q)$. These factors are unchanged in \eqref{eq:odd-family-recursion}. The same expansion and vanishing argument give
\[
 U_r^{(r,w)}(q_p)
 =\gamma_pU_{(p-1)/2}^{(0,w)}(q_p)
\]
with the same scalar $\gamma_p$ as in \eqref{eq:q-proportional}. The initial second coordinate and the recursion are independent of $w$, so the second coordinate agrees with the one for $w=3$ and is nonzero by \cref{prop:q-root}. Taking coordinate ratios proves \eqref{eq:odd-root}.

We next calculate the reduction of the sum in \eqref{eq:odd-E}. Let $p>w$ be prime, put $M=(p-1)/2$, and write
\[
 H_n^{(w)}=\sum_{j=1}^n\frac1{j^w},
 \qquad
 S_w=\sum_{j=1}^{p-1}\frac{(-1)^{j-1}}{j^w}.
\]
Pairing the terms of indices $j$ and $p-j$ gives $S_w\equiv2E_M^{(w)}\pmod p$, since $w$ is odd. Separating the even indices and using $H_{p-1}^{(w)}\equiv0\pmod p$ gives
\[
 S_w
 =H_{p-1}^{(w)}-2^{1-w}H_M^{(w)}
 \equiv-2^{1-w}H_M^{(w)}\pmod p.
\]
The vanishing used here follows from $j^{-w}\equiv j^{p-1-w}\pmod p$ and $1\leq p-1-w\leq p-2$. This gives
\begin{equation}\label{eq:odd-E-H}
 E_M^{(w)}\equiv-2^{-w}H_M^{(w)}\pmod p.
\end{equation}

The Bernoulli polynomial summation formula gives
\[
 H_M^{(w)}
 \equiv
 \frac{B_{p-w}(M+1)-B_{p-w}}{p-w}\pmod p.
\]
By von Staudt--Clausen, the Bernoulli numbers occurring in $B_{p-w}(x)$ are $p$-integral, since their indices are at most $p-w\leq p-3$. Here $M+1\equiv1/2\pmod p$, and
\[
 B_{p-w}(1/2)
 =\left(2^{1-(p-w)}-1\right)B_{p-w}
 \equiv(2^w-1)B_{p-w}\pmod p.
\]
It follows that
\[
 H_M^{(w)}
 \equiv-\frac{2^w-2}{w}B_{p-w}\pmod p.
\]
Combining this with \eqref{eq:odd-E-H}, we obtain
\begin{equation}\label{eq:odd-half-sum}
 E_M^{(w)}
 \equiv\frac{1-2^{1-w}}wB_{p-w}
 =\left(1-2^{1-w}\right)Z(w)_p\pmod p.
\end{equation}
Since $q_p\equiv1$ and $[j]_{q_p}\equiv j$ modulo $1-q_p$, \eqref{eq:odd-half-sum} also gives
\[
 E_M^{(w)}(q_p)
 \equiv\left(1-2^{1-w}\right)Z(w)_p
 \pmod{1-q_p}.
\]
This proves \eqref{eq:odd-algebraic}. The finitely many primes $p\leq w$ do not affect the resulting element of $\cA$.

Finally, the estimate used in the proof of the \cref{thm:main} gives
\[
\left|\frac{q_N^{2j}}{[j]_{q_N}^w}\right|
\leq\left(\frac{\pi}{2j}\right)^w
\qquad\left(1\leq j\leq\frac{N-1}{2}\right).
\]
Extend the summand by zero for $j>(N-1)/2$. For each fixed $j$, $q_N^{2j}/[j]_{q_N}^w\longrightarrow j^{-w}$. Dominated convergence therefore yields
\[
 \lim_{\substack{N\to\infty\\N\ \mathrm{odd}}}
 E_{(N-1)/2}^{(w)}(q_N)
 =\sum_{j\geq1}\frac{(-1)^{j-1}}{j^w}
 =\left(1-2^{1-w}\right)\zeta(w)\,.
\]
By \eqref{eq:odd-normalization}, this is \eqref{eq:odd-analytic}.
\end{proof}

\vspace{1cm}

{\bf AI disclosure:} Some of the proof ideas came from ChatGPT~5.6~Sol, which also assisted with computations and intermediate proof steps together with Claude Fable~5.1. They also helped with the SageMath implementations and language checks. The author checked all proofs and takes full responsibility for the results.

\end{document}